\documentclass[12pt]{amsart}
\usepackage[margin=0.8in]{geometry}
\usepackage[scr=rsfs]{mathalpha}
\usepackage{amsfonts}
\usepackage{amsthm}
\usepackage{amssymb}
\usepackage{times} 
\usepackage{graphicx}
\usepackage{hyperref}
\usepackage{tikz}
\usepackage{comment}
\usetikzlibrary{arrows}
\usepackage[all]{xy}
\usepackage{tikz-cd}

\usepackage{mathtools}

\date{}

\numberwithin{equation}{section}
\newtheorem{thm}{Theorem}[section]

\tikzset{node distance=5cm, auto}

\newtheorem{lemma}[thm]{Lemma}
\newtheorem{conjecture}[thm]{Conjecture}

\newtheorem{remark}[thm]{Remark}
\newtheorem{question}[thm]{Question}

\makeatletter
\newcommand{\etale}{\'etal\@ifstar{\'e}{e\xspace}}
\makeatother
\usepackage[OT2,T1]{fontenc}
\DeclareSymbolFont{cyrletters}{OT2}{wncyr}{m}{n}
\DeclareMathSymbol{\Sha}{\mathalpha}{cyrletters}{"58}

\begin{document}

\title{Totally ramified subfields of $p$-Algebras}
\author{Adam Chapman}
\address{School of Computer Science, Academic College of Tel-Aviv-Yaffo, Rabenu Yeruham St., P.O.B 8401 Yaffo, 6818211, Israel}

\author{S. Srimathy}
\address{School of Mathematics, Tata Institute of Fundamental Research, Mumbai, 400005, India }

\begin{abstract}
 We conjecture that a $p$-algebra over a complete discrete valued field $K$  contains a totally ramified purely inseparable subfield if and only if it contains a totally ramified cyclic maximal subfield.  We prove the conjecture in several cases.
\end{abstract}
\maketitle

\section{Introduction}
A central simple algebra is called a \emph{$p$-algebra} if its center is a field of characteristic $p>0$ and its index is a power of $p$. It is said to be \emph{cyclic} if it contains  a cyclic maximal subfield.  Let $W_m(F)$ denote the truncated Witt vectors of length $m$ over $F$. It is well known (\cite[Chapter VII, \S7, 8]{albert_book}, \cite{MammoneMerkurjev:1991}) that any cyclic $p$-algebra  over a  field $F$ is of the form
\begin{align*}
F \langle x_1,\dots,x_m,y : (x_1^p,\dots,x_m^p)=(x_1,\dots,x_m)+\omega,\\ y^{p^m}=b, (yx_1y^{-1},\dots,y x_m y^{-1})=(x_1,\dots,x_m)+(1,0,\dots,0)\rangle
\end{align*}
where $\omega \in W_m(F)$, $b\in F^*$, and the symbol $+$  above denotes the addition rule of Witt vectors. We denote this algebra by $[\omega,b)_{K}$. \\
\indent It is a famous result of Albert that tensor products of cyclic $p$-algebras are cyclic and moreover, the cyclic $p$-algebras generate the $p$-primary component  of the Brauer group.  In other words, any $p$-algebra is Brauer equivalent to a cyclic $p$-algebra (\cite[Chapter VII, \S9, Theorem 31]{albert_book}).  \\
\indent Purely inseparable maximal subfields play a major role understanding cyclic $p$-algebras. The close connection between cyclic maximal subfield and purely inseparable maximal subfields  is  stated by the main theorem of cyclicity of $p$-algebras: a $p$-algebra is cyclic if and only if it contains a simple purely inseparable maximal subfield (\cite[Chapter IV, Theorem 4.4.10]{jacobson_book}). \\
\indent Now suppose that  $K$ is a  complete discrete valued field  of characteristic $p$ with residue $k$. Given a  $p$-algebra $A$ over $K$,  we analyze the ramification of cyclic maximal subfields and purely inseparable maximal subfields of $A$. We conjecture the following:
\begin{conjecture}\label{conj1}
$A$ contains a totally ramified  cyclic maximal subfield if and only if it contains a totally ramified purely inseparable maximal subfield.
\end{conjecture}

The "only if" direction is proven in \S \ref{Seasy}. The "if" direction is considerably harder. We prove it in each of the following  cases in \S\ref{sec:perfect}, \S\ref{sec:psquare} and \S\ref{sec:residuedim}:
\begin{enumerate}
\item $k$ is perfect
    \item  the degree of $A$ is $p$ or $p^2$,
    \item  $\dim_{\mathbb{F}_p}(k/\mathcal{P}(k)) \geq 2$
    
\end{enumerate}
where $\mathcal{P}(k):= \{a^p-a | a \in k\}$.\\
\indent One may wonder why we consider only totally ramified maximal subfields instead of maximal subfields of given ramification index in the Conjecture \ref{conj1}. We give an explanation below.\\
\indent We first observe  that a totally ramified purely inseparable maximal subfield  of $A$ is necessarily simple  over $K$ (Lemma \ref{lem:simple}) and therefore $A$ is cyclic (\cite[Chapter IV, Theorem 4.4.10]{jacobson_book}). But a purely inseparable extension of $K$ that is not totally ramified need  not necessarily be simple and therefore $A$ need not contain a cyclic maximal subfield.  Conversely,  suppose $k$ is perfect, then note that by the fundamental equality (\cite[Chapter II, \S2, Corollary 1]{serre_local}),  every purely inseparable extension of $K$ is totally ramified.   So in this case, $A$ may contain  cyclic maximal subfields of  ramification index less than the degree of $A$, but $A$ does not contain purely inseparable maximal subfield of the same ramification index.  Therefore the analogous  statement of Conjecture \ref{conj1} is not true for arbitrary complete discrete valued fields if we consider maximal subfields of   ramification index less than the degree of $A$.  See however, Question \ref{quest}. 

\section{Notations}
For a field $F$, the set of non-zero elements of $F$ is denoted by $F^*$.  Given a field extension $L/K$, the symbol $N_{L/K}$ denotes the norm function  of $L$ over $K$. The symbol $W_n(F)$ denotes the truncated Witt vector of length $n$ over $F$. 
 \\
 \indent Throughout this paper $K$ denotes a complete discrete valued field with valuation $\mathfrak{v}$. The residue field of $K$ is denoted $k$.  We denote the value group of $K$ by $\Gamma_K$. For a  finite  extension $L/K$, the (unique) extension of the valuation $\mathfrak{v}$ to $L$ is also denoted by $\mathfrak{v}$.  We normalize the value group of any finite extension of $K$ so that  $\Gamma_K = \mathbb{Z}$.  The greatest common divisor of positive integers $m$ and $n$ is denoted by $(m,n)$.   
\section{Preliminaries}

\subsection{Cyclic extension of degree $p^n$ over fields of characteristic $p$ }\label{sec:cyclic}
Let $F$ be a field with  $\operatorname{char}(F)=p>0$ and let $W_m(F)$ denote the truncated Witt vectors of length $m$. It is well known (see \cite[Chapter III]{Jacobson:1964}, \cite{lara_thesis}) that every cyclic field extension $L$ of $F$ of degree $p^m$  is isomorphic to  $F(x_1, x_2, \cdots x_m)$  where
\begin{align}\label{eqn:cyclicwitt}
(x_1^p,\dots,x_m^p)=(x_1,\dots,x_m)+\omega
\end{align}
for some $\omega \in W_m(F)$ where the $`+'$ symbol is the addition rule in the ring $W_m(F)$. Conversely, every field extension of $F$ of the form  $F(x_1, x_2, \cdots, x_m)$ satisfying (\ref{eqn:cyclicwitt}) is a cyclic extension of degree $p^m$ over $F$. For $\omega \in W_m(F)$, we denote by  $F_{\omega}$, the corresponding cyclic extension of degree $p^m$ over $F$.\\
\indent For example, every cyclic extensions of degree $p^2$ over $F$ corresponds to some Witt vector $ \omega= (\omega_1, \omega_2) \in W_2(F)$ and the corresponding extension $F_{\omega} \cong F(x_1, x_2)$ satisfies
\begin{align}
    x_1^p-x_1&=\omega_1 \label{eqn1:p2}\\ 
    x_2^p-x_2&=\omega_2-\sum_{i=1}^{p-1} \frac{(p-1)!}{i!(p-i)!} \omega_1^k x_1^{p-i} \label{eqn2:p2}
\end{align}
Analogous equations for cyclic extensions of larger degrees are very cumbersome to write down explicitly. \\
\indent We recall that Artin-Schreier extensions of degree $p$ over $F$ are in one-to-one correspondence with the $\mathbb{F}_p^*$ orbits of $F/\mathcal{P}(F)$ where 
\begin{align*}
    \mathcal{P}(F) = \{a^p -a | a \in F\}
\end{align*}
 This is well known. See for example \cite[Remark 2.3]{as_nip}. Moreover, given any  degree $p^m$ cyclic extension  $L$ of  $F$, there always exists a cyclic extension of degree $p^{m+1}$ over $F$ containing $L$ (\cite[Chapter IV, \S4.2, Theorem 4.2.3]{jacobson_book} or \cite[Chapter III, \S5, Theorem 16]{Jacobson:1964}). Since every cyclic extension $L/F$ of degree $p^m$ contains a unique cyclic subextension of degree $p$ over $F$, we get:
\begin{lemma}\label{lem:disjoint}
    Let $F$ be a field of characteristic $p$. Suppose that  $\dim_{\mathbb{F}_p}(F/\mathcal{P}(F)) \geq 2$. Then for any $m\geq 0$, there exists cyclic extensions $L_1/F$ and $L_2/F$ of degree $p^m$ satisfying $L_1 \cap L_2 = F$
\end{lemma}

\subsection{Symbol manipulation in cyclic p-algebras} \label{sec:sym}
The following symbol manipulation techniques can be found in \cite[Proposition 1]{MammoneMerkurjev:1991}.
Let $\omega, \omega' \in W_m(F)$, $a_1, a_2, \dots a_m\in F$, $b,b' \in F^*$. Then in $Br(F)$ we have
\begin{align}
     [\omega, b)_F + [\omega', b)_F &= [\omega + \omega', b)_F\label{item:sym1}\\
      [\omega, b)_F + [\omega, b')_F &= [\omega, bb')_F\label{item:sym2}\\
      [(0, a_1, a_2, \cdots a_m), b)_F &= [(a_1, a_2, \cdots a_m), b)_F \label{item:sym3}\\
      [(b,0 ,0 \cdots, 0), b)_F &= 0 \label{item:sym4}
\end{align}
Moreover by \cite[Chapter 4, Corollary 4.7.5]{GilleSzamuely:2006}, the algebra $[\omega, b)_F$ is split  if and only if 
\begin{align}
    b \in N_{F_{\omega}/F}(F_{\omega}) \label{item:sym5}
\end{align}

In particular, $[\omega, b)_F$ is split if $b \in (F^*)^{p^m}$. 
We also  note that  $F_{\omega}  \cong F_{\omega^{p^r}}$  for any $r$.   Therefore  $[\omega, b)_F \cong [\omega^{p^r}, b]_F$. Similarly,  $[\omega, b)_F \cong [\omega, \gamma^{p^m}b)_F$ for any $\gamma \in F^*$.

\subsection{Ramification of Division algebras over complete discrete valued fields}\label{sec:ramification}
As before, $K$ denotes a complete discrete valued field with valuation $\mathfrak{v}$ and residue $k$ . Let $L$ be a finite extension of $K$ of degree $n$. Then it is well  known that the valuation on $K$ extends uniquely to $L$ (\cite[Chapter II, \S2, Corollary 2, 4]{serre_local}). By abuse of notation, we will also denote the extended valuation on $L$ by $\mathfrak{v}$.  Normalizing  the valuation on $L$ so that $\Gamma_K = \mathbb{Z}$, we get for any $a \in L$,
\begin{align}\label{eqn:extval}
    \mathfrak{v}(a) = \frac{1}{n} \mathfrak{v}(N_{L/K}(a))
    \end{align}

 We will now recall valuation theory on  division algebras. The following are well known can  be found in  \cite{jw_henselian}, \cite{TignolWadsworth:2015} and \cite{wadsworth}. \\ 
 \indent Let $D$ be a division algebra over  $K$. The valuation on $K$ extends uniquely to a valuation on $D$ (\cite[Corolary 2.2]{wadsworth}) also denoted by $\mathfrak{v}$.  Let $\overline{D}$ denote the residue algebra of $D$ and $\Gamma_D$ denote the value group of $D$. Then we have the following fundamental equality:
\begin{align}\label{eqn:fund}
    [D:K] = [\Gamma_D: \Gamma_K] [\overline{D}: k]
    \end{align}
We now observe the following:
\begin{remark}\label{rem:semiramified}
    Suppose $D$ contains an unramified field maximal subfield $L$ as well as a totally ramified maximal subfield (i.e, $D$ is \emph{semiramified} as in \cite[p. 128]{jw_henselian}), then by the fundamental equality (\ref{eqn:fund}), $\overline{D}= \overline{L}$.
\end{remark}
The following lemma gives a sufficient condition for a $p$-algebra over $K$ to be a division algebra.
\begin{lemma}\label{lem:division}
    Let $b \in K^*$ with $(\mathfrak{v}(b),p) = 1$. Let $\omega \in W_m(K)$ be such that $K_{\omega}/K$ is unramified.  Then  $[\omega, b)_K$ is a division algebra over $K$.
\end{lemma}
\begin{proof}
     Suppose  there exists  $r> 0$ such that $r$ is the least number satisfying 
    \begin{align*}
        b^r = N_{K_{\omega}/K}(\alpha)
    \end{align*}
    for some $\alpha \in K_{\omega}$.  Now,  comparing the valuations on both sides of the equation and using (\ref{eqn:extval}), we get
  \begin{align*}
      r\mathfrak{v}(b) = p^m\mathfrak{v}(\alpha)
  \end{align*}
Since $K_{\omega}/K$ is unramified, $\mathfrak{v}(\alpha) \in \Gamma_K = \mathbb{Z}$. By hypothesis, $(\mathfrak{v}(b),p) = 1$ and hence $p^m$ divides $r$. The claim now follows from \cite[p. 166, last line]{wedderburn_primitive}. 
\end{proof}

\section{From cyclic to purely inseparable}\label{Seasy}

\begin{thm}
Let $A$  be a cyclic algebra of degree $p^m$ over $K$. Suppose $A$ contains a  totally ramified cyclic maximal subfield, then $A$ contains a totally ramified  purely inseparable maximal subfield.
\end{thm}

\begin{proof}
    Let $\omega \in W_m(K)$ correspond to a totally ramified cyclic maximal subfield of $A$. By \cite[Chapter VII, Theorem 26]{albert_book}, there exists $b \in K^*$ such that $A \cong  [\omega, b)_{K}$.  If $(\mathfrak{v}(b),p)=1$ then $\mathfrak{v}(\sqrt[p^m]{b}) \in \frac{1}{p^m} \mathbb{Z} \setminus \frac{1}{p^{m-1}} \mathbb{Z}$. Hence $K(\sqrt[p^m]{b})$ is a totally ramified maximal subfield of $A$. Otherwise, pick an element $u \in K_{\omega}$ with $\mathfrak{v}(u) \in \frac{1}{p^m} \mathbb{Z} \setminus \frac{1}{p^{m-1}} \mathbb{Z}$. Note that $\mathfrak{v}(N_{K_{\omega/K}}(u)) \in \mathbb{Z} \setminus p\mathbb{Z}$. \\
    \indent Consider the element $z=uy$. We have 
    \begin{align*}
        z^{p^m} = N_{K_{\omega/K}}(u)b \in K^*
    \end{align*}
     Since $\mathfrak{v}(b) \in p\mathbb{Z}$,  $\mathfrak{v}(z) \in \frac{1}{p^m} \mathbb{Z} \setminus \frac{1}{p^{m-1}} \mathbb{Z}$ and hence $p^m$ is the minimal power of $z$ that lives in $K^*$. Therefore, $K(z)$ is a  totally ramified purely inseparable maximal subfield of $A$.
\end{proof}

\section{From purely inseparable to cyclic}
Now we show that the existence of totally ramified purely inseparable maximal subfield in $A$ implies  the existence of totally ramified cyclic maximal subfield for various cases. We first observe that a totally ramified purely inseparable extension of $K$ is necessarily simple:

\begin{lemma}\label{lem:simple}
    Let $L/K$ be a  totally ramified purely inseparable extension  of degree $p^m$. Then $L/K$ is simple, generated by $y$ satisfying $y^{p^m}=b \in K^*$ with $(\mathfrak{v}(b), p ) =1$.
\end{lemma}
\begin{proof}
   Since $L/K$ is  totally ramified purely inseparable extension, there exists an element $y \in L$ such that $\mathfrak{v}(y) \in \frac{1}{p^m} \mathbb{Z}\setminus \frac{1}{p^{m-1}} \mathbb{Z}$. Since the extension is purely inseparable, $y^{p^m} = b \in K^*$. Note that $p^m$  is the smallest power of $y$ that lives in $K^*$ since $\mathfrak{v}(z) \in \frac{1}{p^m} \mathbb{Z}$. Therefore, $y$ generates $L$ over $K$ and satisfies $(\mathfrak{v}(b), p)=1$.
\end{proof}
\begin{remark}\normalfont\label{rem:general}
   Suppose a $p$-algebra $A$  over $K$ of degree $p^m$ contains a totally ramified purely inseparable maximal subfield. By the above lemma, it is of the form $K(\sqrt[p^m]{b})$ where $(\mathfrak{v}(b), p)= 1$. Therefore, $A$ is necessarily cyclic and  $A\cong [\omega, b)_K$ for some $\omega \in W_m(K)$ (\cite[Chapter IV, Theorem 4.4.10]{jacobson_book}, \cite[Chapter VII, Theorem 26]{albert_book}).
\end{remark}

We prove the following main theorem of this paper.
\begin{thm}\label{thm:main}
    Let $A$ be a  $p$-algebra over a complete discrete valued field $K$ with residue $k$. Suppose $A$ contains a totally ramified purely inseparable maximal subfield. Then $A$ contains a totally ramified cyclic maximal subfield in each of the  following cases:
    \begin{enumerate}
        \item $k$ is perfect,
        \item the degree of $A$ is $p$ or $p^2$,
        \item  $\dim_{\mathbb{F}_p}(k/\mathcal{P}(k)) \geq 2$.
    \end{enumerate}
\end{thm}

   The proof is demonstrated in the following sections. See \S\ref{sec:perfect}, \S\ref{sec:psquare} and \S\ref{sec:residuedim}.

\subsection{When $k$ is perfect}\label{sec:perfect}
In this section, we assume that the residue field $k$ of $K$ is perfect.  
\begin{thm}
    Let $A$ be a  $p$-algebra over $K$ whose  residue field $k$ is perfect.  Suppose $A$ contains a totally ramified purely inseparable maximal subfield, then  it contains a totally ramified cyclic maximal subfield.
\end{thm}
\begin{proof}
    The proof is similar to the proof of  \cite[Appendix B, Theorem B.1]{srimathy2024genus} with some  modifications.  Let the degree of $A$ be $p^m$. By Remark \ref{rem:general}, $A \cong [\omega, b)$  where $\omega  = (\omega_1, \omega_2, \cdots, \omega_m) \in W_m(K)$  and  $(v(b),p) =1$.  By \S\ref{sec:sym}, we may assume that $v(b)<min(0, v(\omega_1))$. Using (\ref{item:sym1}) and (\ref{item:sym4}), we get 
\begin{align*}
    A \cong [\omega', b):=[(\omega + (b,0,0 \cdots 0) , b)]
     \end{align*}
 where $ \omega' = (\omega_1 +b, \omega_2', \cdots, \omega_m') \in W_m(K)$ for some $\omega_2', \cdots, \omega_m' \in K$. The cyclic extension  $K_{\omega'}$  is given by $K_{\omega'} \cong  K(x_1, x_2, \cdots, x_m)$ where 
 \begin{align*}
     (x_1^p, x_2^p, \cdots, x_m^p)- (x_1, x_2, \cdots, x_m) = \omega'
 \end{align*}  
 in  $W_m(K_{\omega'})$. We claim that $K_{\omega'}$ is totally ramified which will yield the theorem. To see this, let $k_{\omega'}$ denote the residue field of $K_{\omega'}$. Suppose $K_{\omega'}/K$ is not totally ramified, then   $k_{\omega'}$ is a non-trivial extension of $k$. Since $k$ is perfect,  $k_{\omega'}/k$ is separable.  Also, since $K_{\omega'}/K$ is cyclic, so is $k_{\omega'}/k$ (\cite[Chapter III, \S5, Theorem 3]{serre_local})).  Let $f \subset k_{\omega'}$ be the unique degree $p$ subextension over $k$ and let $F \subset K_{\omega'}$ be the inertial lift  of $f$ \cite[Chapter III, \S5, Corollary 2]{serre_local}). Therefore $F$ is an unramified  degree $p$ extension of $K$ inside $K_{\omega'}$. But $K_{\omega'}/K$ is cyclic and contains the unique degree $p$ extension $k(x_1)$ defined by 

\begin{align*}
x_1^p - x_1 = \omega_1 + b
\end{align*}
with $v(\omega_1 + b) < 0$ and  $v(x_1) = \frac{1}{p}\mathbb{Z}$ by the assumption on $b$. Hence $k(x_1)$ is ramified leading to contradiction. Therefore $K_{\omega'}/K$ is totally ramified.  
\end{proof}

\subsection{When the degree of $A$ is $p$ or  $p^2$} \label{sec:psquare}
From now on we do not assume that the residue field $k$ of $K$ is perfect. In this case, when the degree of $A$ is $p$, the proof of the claim involves a neat argument on the level of the generators.

\begin{thm}
Let $A$ be a $p$-algebra of degree $p$ over a complete discrete valued field $K$ that contains a  totally ramified purely inseparable maximal subfield, then $A$ contains a totally ramified cyclic  maximal subfield.
\end{thm}

\begin{proof}
 In this case, $A=[\omega,b)_K$ for some $\omega \in K$ and $b\in K^*$. By Remark \ref{rem:general}, we can take $(\mathfrak{v}(b), p ) =1$. Moreover, we can assume that $\mathfrak{v}(b)<\mathfrak{v}(\omega)$ (\S \ref{sec:sym}).  Now
    \begin{align*}
        A \cong K\langle x,y : x^p-x=\omega, y^p=b, yxy^{-1}=x+1\rangle
    \end{align*}
    Consider the element $z=x+y$. Clearly $\mathfrak{v}(z) = \mathfrak{v}(y) = \frac{1}{p}\mathbb{Z}$. Moreover, $z$ satisfies $z^p-z=\omega+b$ (see \cite[lemma 3.1]{Chapman:2015}), and therefore $K(z)$ is a totally ramified cyclic maximal subfield of $A$.
\end{proof}

Now assume that the degree of $A$ is $p^2$. 
Let $A$ contain a totally ramified purely inseparable maximal subfield $L$. By Lemma \ref{lem:simple},   $L \cong K(\sqrt[p^2]{b})$  with $(v(b),p) = 1$. 
By Remark \ref{rem:general}, there exists $\mathbf{\omega} = (\omega_1, \omega_2) \in W_2(K)$ such that  $A \cong [\mathbf{\omega}, b)$.  Let us start with a few lemmas.\\

\begin{lemma} \label{lem:eta}
Let $\mathbf{\eta}=(\eta_1,\eta_2) \in W_2(K)$. If $\mathfrak{v}(\eta_1)<0$, $\mathfrak{v}(\eta_1)<\mathfrak{v}(\eta_2)$ and $\mathfrak{v}(\eta_1)\notin p\mathbb{Z}$, then  $K_{\eta}$ is totally ramified. 
\end{lemma}

\begin{proof}
The field $K_{\eta}$ is defined by the equations (see (\ref{eqn1:p2}) and (\ref{eqn2:p2}))
\begin{align}
    x_1^p-x_1&=\eta_1 \label{eqn:x1}\\ 
    x_2^p-x_2&=\eta_2-\sum_{i=1}^{p-1} \frac{(p-1)!}{i!(p-i)!} \eta_1^k x_1^{p-i}\label{eqn:x2}
\end{align}
From (\ref{eqn:x1}) we deduce that since  $\mathfrak{v}(\eta_1)<0$, the value of $x_1$ must be negative, and thus $\mathfrak{v}(x_1^p-x_1)=\mathfrak{v}(x_1^p)=p\mathfrak{v}(x_1)=\mathfrak{v}(\eta_1)$. Therefore
\begin{align*}
  \mathfrak{v}(x_1)=\frac{1}{p} \mathfrak{v}(\eta_1)  \in \frac{1}{p}\mathbb{Z}
\end{align*}
 Analyzing (\ref{eqn:x2}), we see that the term of smallest value on the right-hand side is $-\eta_1^{p-1} x_1$. Therefore
 \begin{align*}
 \mathfrak{v}(x_2)=\frac{1}{p}((p-1)\mathfrak{v}(\eta_1)+\mathfrak{v}(x_1)) \in \frac{1}{p^2}\mathbb{Z}
 \end{align*}
 and the statement thus follows.
\end{proof}

\begin{lemma}\label{lem:split}
 The algebra  $[(0,rc^{pi} b^{p-i}), b)_K $ is split     for all $c \in K, b\in K^*$,  $r,i\in \mathbb{Z}$.
\end{lemma}
\begin{proof}
By  (\ref{item:sym3}), $[(0,rc^{pi} b^{p-i}), b)_K  = [rc^{pi} b^{p-i}, b)_K $ in $Br(K)$. Now,
\begin{align*}
    [rc^{pi}b^{p-i},b) &\cong[(p-i)^{-1} rc^{pi}b^{p-i},b^{p-i}) \text{~~by~(\ref{item:sym1})~and~(\ref{item:sym2})~}\\
    &\cong [(p-i)^{-1} rc^{pi}b^{p-i},(p-i)r^{-1}c^{-pi})  \text{~~by~(\ref{item:sym2})~and~(\ref{item:sym4})}\\
    & = 0 \in Br(K) \text{~~by~(\ref{item:sym5})~}
    \end{align*}

\end{proof}
\begin{lemma}\label{lem:add}
If $\omega_1 \in K^p$, then $[(\omega_1,\omega_2),b)_{K} \cong [(\omega_1+b,\omega_2),b)_{K}$.
\end{lemma}

\begin{proof}

Let $\omega_1 = c^p, c \in K$. By the addition rule in $W_2(K)$,
\begin{align*}
    (\omega_1,\omega_2)+(b,0)=(\omega_1+b,\omega_2-\sum_{i=1}^{p-1} \frac{(p-1)!}{i!(p-i)!} c^{pi} b^{p-i})
\end{align*}

By Lemma \ref{lem:split},  (\ref{item:sym1}) and (\ref{item:sym4}), $[(\omega_1,\omega_2),b)_{K} \cong [(\omega_1',\omega_2'),b)_{K}$ where 
\begin{align*}
    (\omega_1',\omega_2') &= (\omega_1,\omega_2) +(b,0) + \sum_{i=1}^{p-1}(0, \frac{(p-1)!}{i!(p-i)!} c^{pi} b^{p-i})\\
    &= (\omega_1 +b, \omega_2)
\end{align*}

\end{proof}

\begin{thm}
Let $A$ be  $p$-algebra of degree $p^2$ over $K$. If $A$ contains a totally ramified purely inseparable maximal subfield, then $A$ contains a totally ramified cyclic maximal subfield.
\end{thm}

\begin{proof}
By Remark \ref{rem:general}, $A\cong [(\omega_1, \omega_2), b)$ where $(v(b),p) =1$. Moreover,  we can assume that  $\omega_1,\omega_2 \in K^p$ and $v(b)<\min\{0, v(\omega_1), v(\omega_2)\}$ (\S\ref{sec:sym}). Now, by Lemma \ref{lem:add}, $A \cong [(\omega_1+b,\omega_2),b)_{K}$.
The Witt vector $\omega' = (\omega_1+b,\omega_2)$ satisfies the conditions of Lemma \ref{lem:eta}, and thus $K_{\omega'}$ is a totally ramified cyclic maximal subfield of $A$.
\end{proof}

\subsection{When $\dim_{\mathbb{F}_p}(k/\mathcal{P}(k)) \geq 2$}\label{sec:residuedim}

The above techniques do not extend easily to $p$-algebras of degree $p^m$ for $m\geq 3$, due to the complexity of the terms involved in the equations defining the cyclic extensions arising from Witt vector addition. Due to this reason, analyzing the ramification  of  these extensions become very cumbersome. We therefore turn to a different direction.  \\
\indent In this section, we will assume that the residue field $k$ satisfies $\dim_{\mathbb{F}_p}(k/\mathcal{P}(k)) \geq 2$ so that there are at least two  linearly disjoint cyclic field extensions of degree $p^m$ over $k$ for every $m>0$.   We will use  the following lemma.

\begin{lemma}[{\cite[Theorem 4.7]{CFM}}]\label{Clem}
Suppose $F$ any field with $char(F)=p>0$. If $A$ and $B$ are two $p$-algebras over $F$ sharing a simple purely inseparable maximal subfield, then they share a cyclic maximal subfield.
\end{lemma}

\begin{thm}\label{main}
    Let $A$ be a $p$-algebra over $K$ with  totally ramified purely inseparable maximal subfield. Suppose the residue field $k$ satisfies $\dim_{\mathbb{F}_p}(k/\mathcal{P}(k)) \geq 2$,   then $A$ contains a totally ramified cyclic maximal subfield.
\end{thm}

\begin{proof}
   Let the degree of $A$ be $p^m$. By Lemma \ref{lem:simple},  $A$ contains a  totally ramified purely inseparable extension of degree $p^m$ of the form  $K(\sqrt[p^m]{b})$, for some $b\in K^*$ with $(\mathfrak{v}(b), p) = 1$.  
  Since $\dim_{\mathbb{F}_p}(k/\mathcal{P}(k)) \geq 2$ there exists two  cyclic field extensions  $k_1/k$ and $k_2/k$ of degree $p^m$ with $k_1\cap k_2 = k$ (Lemma \ref{lem:disjoint}). Let $K_1, K_2$ be the respective inertial lifts over $K$ corresponding to $\omega_1, \omega_2 \in W_m(K)$. Then by Lemma \ref{lem:division}, $D_i = [\omega_i, b)_K$ are   division algebras over $K$ for $i=1,2$. Moreover, $\overline{D_i} = k_i, i=1,2$ (Remark \ref{rem:semiramified}). Now $A$, $D_1$ and $D_2$ share  the same purely inseparable subfield $K(\sqrt[p^m]{b})$ and therefore by Lemma \ref{Clem}, share a cyclic maximal subfield $L$. Now 
  \begin{align*}
  \overline{L} \subseteq \overline{D_1} \cap \overline{D_2} = k_1 \cap k_2 = k
\end{align*}  
 Hence $L/K$ is totally ramified as required.
    \end{proof}


\section{Existence of totally ramified separable maximal subfields}

Suppose a $p$-algebra $A$ contains  a totally ramified maximal subfield $L$, then $L/K$ is a tower of extensions $L_1/K$ and $L_2/L_1$, the former separable and the latter purely inseparable (\cite[\href{https://stacks.math.columbia.edu/tag/030K}{Tag 030K}]{stacks-project}). Let $B:=C_A(L_1)$ be the centralizer of $L_1$ in $A$. Then by the Double Centralizer Theorem (\cite[\S12.7 Theorem (ii)]{pierce_assoc}), $B$ contains the totally ramified purely inseparable maximal subfield $L_2$.  Assume that Conjecture \ref{conj1} is true. Then $B$ contains a  totally ramified cyclic maximal subfield $E$. Now $E/L_1/K$ is tower of  totally ramified separable extensions in $A$ and $E$ is  maximal subfield of $A$. We summarize this below:
\begin{thm}
   Assume that Conjecture \ref{conj1} is true. Then any $p$-algebra with a totally ramified maximal subfield contains a totally ramified separable maximal subfield. 
\end{thm}

We finish by asking the following question.
\begin{question}\label{quest}
Let $A$ be  $p$-algebra over $K$. Assume that the residue field $k$ is not perfect. If $A$ contains purely inseparable maximal subfield of ramification index $r$  then  does it  contain cyclic maximal subfield of ramification index $r$?
\end{question}

\section*{Acknowledgements}
The authors wish to thank David Saltman and Jean-Pierre Tignol for the illuminating discussions on this topic. The second author acknowledges the support of the DAE, Government of India, under Project Identification No. RTI4001.

\bibliographystyle{abbrv}
\bibliography{bibfile}

\end{document}